\documentclass[a4paper,12pt]{amsart} 
\usepackage{etex}
\usepackage{pgf}
\usepackage{amscd,amsmath,amssymb,amstext,amsfonts,amsthm,latexsym}
\usepackage{verbatim}
\usepackage[english]{babel}
\usepackage[latin1]{inputenc}
\usepackage{graphicx}
\usepackage{mathrsfs,ulem}
\usepackage{setspace}

\usepackage[all]{xy}
\usepackage{supertabular}
\usepackage{fancyhdr}
\usepackage{ url}
\usepackage{setspace}

 \usepackage[usenames,dvipsnames]{pstricks}
 \usepackage{epsfig}
 \usepackage{pst-grad} 
 \usepackage{pst-plot} 

\usepackage{hyperref}

\newtheorem{thm}{Theorem}[section]
\newtheorem*{thm2}{Theorem}
\newtheorem*{thmA}{Theorem A}
\newtheorem*{thmB}{Theorem B}
\newtheorem*{thmC}{Theorem C}
\newtheorem{lem}[thm]{Lemma}
\newtheorem{cor}[thm]{Corollary}
\newtheorem{prop}[thm]{Proposition}

\theoremstyle{remark}
\newtheorem{rem}[thm]{Remark}

\theoremstyle{definition}
\newtheorem{defi}[thm]{Definition}

\newcommand{\bbC}{{\mathbb{C}}}

\newcommand{\bbQ}{{\mathbb{Q}}}
\newcommand{\bbZ}{{\mathbb{Z}}}

\newcommand{\bbP}{{\mathbb{P}}}

\newcommand{\BB}{\mathcal{B}}
\newcommand{\OO}{\mathcal{O}}
\newcommand{\Sf}{\mathfrak{S}}

\newcommand{\im}{\mathrm{Im}}

\newcommand{\Aut}{\mathrm{Aut}}
\newcommand{\Sing}{\mathrm{Sing}}

\newcommand{\ord}{\mathrm{ord}}
\newcommand{\Stab}{\mathrm{Stab}}

\newcommand{\lcm}{\mathrm{lcm}}

\newcommand{\alb}{\mathrm{Alb}}
\newcommand{\ud}{\mathrm{d}}	
\newcommand{\GL}{\mathrm{GL}}

\newcommand{\st}{\textrm{such that}} 
\newcommand{\sts}{\textrm{such that }}

\newcommand{\no}[1]{\textit{#1}}
\newcommand{\ov}[1]{\overline{#1}}
\newcommand{\q}[1]{``{#1}''}
\newcommand{\tr}[1]{\langle{#1}\rangle}

\newcommand{\ie}{{i.e.}\ }
\newcommand{\qe}{{q.e.}\ }
		
\thanks{Some of the results in this paper were developed in fall 2012, when the first author was supported by the Department of Mathematics of the University of Trento. 
The second author was partially supported by the FCT Project PTDC/MAT/111332/2009 {\it Moduli Spaces in Algebraic Geometry}, by the PRIN 2010-2011 Project {\it Geometria delle variet\`a algebriche} and by the FIRB 2012 Project {\it Spazi di moduli e applicazioni}.  
Both authors  are members of G.N.S.A.G.A. of I.N.d.A.M}
																																
\title{Mixed quasi-\'etale quotients with arbitrary singularities}
\makeatletter

\@addtoreset{equation}{section}
\makeatother
\author[D. FRAPPORTI, R. PIGNATELLI]{DAVIDE FRAPPORTI, ROBERTO PIGNATELLI}

\date{\today}

\begin{document}

\begin{abstract} 
A mixed quasi-\'etale quotient is the quotient of the product of a curve of genus at least $2$ with itself by the action of a group which exchanges 
the two factors and acts freely outside a finite subset.  A mixed quasi-\'etale surface is the minimal resolution of its singularities. 

We produce an algorithm computing all mixed quasi-\'etale surfaces with given geometric genus, irregularity, and self-intersection of the canonical class.  
We prove that all irregular mixed  quasi-\'etale surfaces of general type are minimal. 

 As an application, we classify all irregular mixed quasi \'etale surfaces of general type with genus equal to the irregularity, and all the regular ones with $K^2>0$,
 thus constructing new examples of surfaces of general type with $\chi=1$. We mention the first example of a minimal surface of general type  with $p_g=q=1$ 
and Albanese fibre of genus bigger than $K^2$. 
\end{abstract}
\maketitle

\section*{Introduction}


In the last decade, after the seminal paper \cite{Cat00}, there has been growing interest in those surfaces
birational to the quotient of the product of two curves of genus at least $2$ by the action of a 
subgroup of its automorphism group. 

These have shown to be a very productive source of examples, especially in the very interesting and still mysterious 
case of the surfaces of general type with $\chi(S)=1$ (equivalently $p_g(S)=q(S)$). Here and in the following we use the standard notation of the 
theory of the complex surfaces, as in \cite{Beau, BHPV}. For motivation and for the state of the art (few years ago)
of the research on the surfaces of general type with $p_g=q=0$ we suggest to the reader the survey \cite{Survey}, 
while some information on the more general case $\chi(S)=1$ can be found in \cite[Section 2]{CSGT}. 
We just mention here that the case $p_g=q \geq 3$ has been classified (\cite{be82, CCML98, Pir02, HP02}), whereas the case $p_g = q \leq 2$ is still rather unknown.

Recently several new surfaces of general type with $p_g=q$ have been constructed as the quotient of a product of two curves by the action of a finite group; see
\cite{BC04, BCG08, BCGP08, BP10, BP13} for $p_g=0$, 
\cite{CP09, Pol07, Pol09, MP10} for $p_g=1$,
\cite{Penegini, zuc} for $p_g=2$.  
In all these articles the authors assume either that the action is free, or {\it unmixed}, 
which means that the action is diagonal, induced by actions on the factors. 

In \cite{Frap11} the first author considered a more general case,  assuming the action 
to be free outside a finite set of points: it is not difficult to show that this includes both the cases above. 
We call this case {\it quasi-\'etale} since the induced map into the quotient is quasi-\'etale 
in the sense of \cite{Cat07}. Since the above mentioned papers give a satisfactory description of the {\it unmixed} case,
\cite{Frap11} concentrated on the {\it mixed} case, which is the complementary case. After some preliminary results, 
\cite{Frap11} restricted to the case of surfaces of general type with $p_g=0$, and imposed a strong condition on the singularities of the quotient surface,
 obtaining several interesting new examples.

In this paper we drop any assumption on the value of $p_g$ and the type of singularities.

The situation is the following. Let $C$ be a Riemann surface of genus $g(C)\geq 2$, and let $G$ be a finite group that 
acts on $C\times C$. We say that $X=(C\times C)/G$ is a \no{quasi-\'etale quotient} if the action of $G$ is free outside a finite set of points. 
Let $S\rightarrow X$ be the minimal resolution of the singularities of $X$, we call $S$ a \no{quasi-\'etale surface}. The action is 
 \no{mixed} if $G\subset \Aut(C\times C)\cong \Aut(C)^2\rtimes \bbZ_2$
 is not contained in $\Aut(C)^2$; if the action is mixed we say that $X$ is a \no{mixed \qe quotient}, $S$ is a {\it mixed \qe surface} and
we denote by $G^0\triangleleft G$ the subgroup $G \cap \Aut(C)^2$.

The main result of this paper is an algorithm which given three fixed integers $p_g$, $q$ and $K^2$, produces all mixed \qe surfaces with those invariants. 
We implemented the algorithm in the program MAGMA \cite{MAGMA}; the script is available from
$$\mbox{\url{http://www.science.unitn.it/~pignatel/papers/Mixed.magma}}\,$$ 

As an application, running the program for all possible positive values of $K^2$ and $p_g=q$,  we obtained the following theorems A, B and C. 
Note that the program also works for arbitrary values of $K^2$, $p_g$ and $q$, so more surfaces may be produced with it.

\begin{thmA}
The mixed \qe surfaces $S$ with $p_g=q=0$ and $K^2>0$ form the 17 irreducible families collected in Table \ref{tabBigInt0}.
In all cases $S$ is minimal and of general type.
\end{thmA} 
\begin{table*}[!ht]
\centering{\footnotesize
\begin{tabular}{|c|c|c|c|c|c|c|}
\hline

$K^2_S$ & $\BB(X)$ & Sign. & $G^0$ & $G$  & $H_1(S,\bbZ)$ & $\pi_1(S)$ \\

\hline\hline

1	& $2 \, C_{2,1},2\,  D_{2,1}$ & $2^3\,,4$ & $D_4\times  \bbZ_2$ & $\bbZ_2^3\rtimes \bbZ_4$ & $\bbZ_4$ & $\bbZ_4$   \\

\hline\hline

2 & $6 \, C_{2,1}$ & $2^5$ & $\bbZ_2^3$ & $\bbZ_2^2\rtimes \bbZ_4$& $\bbZ_2\times\bbZ_4$ & $\bbZ_2\times\bbZ_4$ \\

2 & $6 \, C_{2,1}$ & $4^3$ & $(\bbZ_2\times \bbZ_4)\rtimes\bbZ_4$ & G(64,82) & $\bbZ_2^3$&$\bbZ_2^3$\\

2 & $  C_{2,1},2\, D_{2,1}$ & $2^3\,,4$ & $\bbZ_2^4\rtimes \bbZ_2$ & $\bbZ_2^4\rtimes \bbZ_4$  & $\bbZ_4$ & $\bbZ_4$\\

2 & $  C_{2,1},2\,  D_{2,1}$ & $2^2\,,3^2$ & $\bbZ_3^2\rtimes \bbZ_2$& $\bbZ_3^2\rtimes \bbZ_4$& $\bbZ_3$ &$\bbZ_3$\\

2 & $2 \, C_{4,1},3 \, C_{2,1}$ & $2^3\,,4$ & G(64,73) &  G(128,1535) &$\bbZ_2^3$ &$\bbZ_2^3$  \\

2 & $2 \, C_{3,1},2\,  C_{3,2}$ & $3^2\,,4$ & G(384,4) &  G(768,1083540) &$\bbZ_4$ &$\bbZ_4$\\

2 & $2 \, C_{3,1},2 \, C_{3,2}$ & $3^2\,,4$ & G(384,4) &  G(768,1083541)& $\bbZ_2^2$ & $\bbZ_2^2$   \\
\hline\hline

3 & $  C_{8,3},  C_{8,5}$ & $2^3\,,8$ & G(32, 39) &  G(64, 42)  & $\bbZ_2\times \bbZ_4$& $\bbZ_2\times \bbZ_4 $ \\
 
\hline\hline

4 & $4 \, C_{2,1}$ & $2^5$ & $D_4 \times \bbZ_2$& $D_{2,8,5}\rtimes \bbZ_2$ &  $\bbZ_2\times \bbZ_8$ &$\bbZ_2^2\rtimes \bbZ_8$   \\

4 & $4 \, C_{2,1}$  & $2^5$ & $\bbZ_2^4$ & $(\bbZ_2^2\rtimes \bbZ_4) \times \bbZ_2$  & $\bbZ_2^3\times \bbZ_4$& $\infty$    \\

4 & $4 \, C_{2,1}$  & $4^3$ &  G(64, 23) & G(128, 836) & $\bbZ_2^3$& $\bbZ_4^2\rtimes \bbZ_2$ \\

\hline\hline

8 & $\emptyset$ & $2^5$ & $D_4 \times \bbZ_2^2$ & $(D_{2,8,5}\rtimes \bbZ_2) \times \bbZ_2$  &$\bbZ_2^3\times \bbZ_8$
& $\infty$ \\

8 & $\emptyset$ & $4^3$ & G(128, 36) & G(256, 3678)&  $\bbZ_4^3$  & $\infty$   \\

8 & $\emptyset$ & $4^3$ & G(128, 36) & G(256, 3678)&  $\bbZ_2^4\times \bbZ_4$  &$\infty$ \\

8 & $\emptyset$ & $4^3$ & G(128, 36) & G(256, 3678)& $\bbZ_2^2\times \bbZ_4^2$  & $\infty$\\

8 & $\emptyset$ & $4^3$ & G(128, 36) & G(256, 3679)& $\bbZ_2^2\times \bbZ_4^2$  & $\infty$ \\
\hline

\end{tabular}}
\caption{Mixed \qe surfaces of general type with $K^2>0$ and $p_g=q=0$ }
\label{tabBigInt0}
\end{table*}
In Table \ref{tabBigInt0}, every row corresponds to an irreducible family. Two columns need some explanation: 
the column $\BB(X)$ represents the basket of singularities of $X$ (see Definition \ref{defbasket}), 
the column Sign.\! gives the signature of the generating vector of $G^0$ (see Definition \ref{gv})
 in a compact way, {\it e.g.} $2^3,4$ stands for $(q;2,2,2,4)$.
  Throughout  the paper 
  we denote by $\bbZ_n$ the cyclic group of order $n$, by
  $\Sf_n$ the symmetric group on $n$ letters, by $A_n$ the alternating group on $n$ letters,
  by $Q_8$ the group of quaternions, by $D_n$ the 	dihedral group of order $2n$,
  by	 $D_{p,q,r}$ the group $\tr{x,y\mid x^p= y^q=1, xyx^{-1}=y^r}$,
  by $BD_n$ the group $\tr{x,y\mid y^{2n}=x^2y^n=1,  xyx^{-1}= y^{-1}}$  
	  and by G(a,b) the $b^{th}$ group of order $a$ in the MAGMA database of finite group. 
	  
	  Note that the $13$ rows of this list where the mixed q.e. quotient has only Rational Double Points as singularities were already in \cite{Frap11}, so only $4$ of these surfaces are new: three with $K^2=2$ and one with $K^2=3$.
Note moreover that some of these surfaces have the same invariants of some of the surfaces in \cite{BP10}, including the fundamental group (see, {\it e.g.}, the case $K^2=4$). We do not know if two such surfaces are deformation equivalent or not: it would be interesting to study their moduli spaces.

There may exist more mixed \qe surfaces of general type with $p_g=q=0$: they would have $K^2 \leq 0$ and therefore they would not be  minimal. 
The strategy (and the program) works in principle for every value of $K^2$. 
Unfortunately, if it seems to work in the case $K^2=0$ (we have not completed the computations), the case when $K^2$ is negative seems to be too hard for our program: indeed we tried on the best computers at our disposal, but it ran out of memory very quickly.

In the irregular ($q>0$) case, the situation is, from this point of view, much more clear since we could prove the following Theorem (\ref{minirr}).
\begin{thm2}
Let $S$ be an irregular mixed \qe surface of general type, then $S$ is minimal. 
\end{thm2}
The result does not extend to the unmixed case, counterexamples can be found  in \cite{MP10}. 
Then we could give a complete classification of the mixed \qe irregular surfaces of general type with $p_g=q$.
 
\begin{thmB}
The mixed \qe surfaces of general type $S$ with $p_g=q=1$  form the 19 irreducible families collected in Table \ref{tabBigInt1}.
\end{thmB} 

	\begin{table*}[!ht]
\centering{\footnotesize
\begin{tabular}{|c|c|c|c|c|c|c|}
\hline

$K^2_S$ & $g_{alb}$ & $\BB(X)$ & Sign. & $G^0$ & $G$  & $H_1(S,\bbZ)$  \\
\hline\hline

2 & 2 & $  C_{2,1}, 2 D_{2,1} $ &$2^2$& $\bbZ_2$&$\bbZ_4$& $ \bbZ^2$\\
2 & 2 & $  C_{2,1}, 2 D_{2,1} $ &2& $D_8$ &$D_{2,8,3}$& $ \bbZ^2$\\
2 & 2 & $  C_{2,1}, 2 D_{2,1} $ &2& $Q_8$ & $BD_4$ & $ \bbZ^2$\\
\hline\hline
4 & 3 & $ 4  C_{2,1} $ &$2^2$& $\bbZ_4$&$\bbZ_8$& $\bbZ_2\times \bbZ^2$\\
4 & 3 & $ 4  C_{2,1} $ &$2^2$& $\bbZ_2\times\bbZ_2$ & $\bbZ_2\times \bbZ_4$& $\bbZ_2\times \bbZ^2$\\
4 & 2 & $ 4  C_{2,1} $ &2& $\bbZ_2^2\rtimes \bbZ_4 $&G(32,29)& $\bbZ_2^2\times \bbZ^2$\\
4 & 3 & $ 4  C_{2,1} $ &2& $D_{4,4,3}$ & $D_{4,8,3}$& $\bbZ_2\times \bbZ^2$\\
4 & 3 & $ 4  C_{2,1} $ &2& $D_{4,4,3}$ &$D_{4,8,7}$& $\bbZ_2\times \bbZ^2$\\
4 & 2 & $ 4  C_{2,1} $ &2& $D_{4,4,3}$ &G(32,32)& $\bbZ_2^2\times \bbZ^2$\\
4 & 2 & $ 4  C_{2,1} $ &2& $D_{4,4,3}$ &G(32,35)& $\bbZ_2^2\times \bbZ^2$\\
4 & 3 & $ 4  C_{2,1} $ &2& $D_{2,8,5}$ &G(32,15)& $\bbZ_2\times \bbZ^2$\\
\hline\hline
5 & 3 & $ C_{3,1} , C_{3,2}$ &3& $BD_3$&$BD_6$& $\bbZ_2\times \bbZ^2$\\
5 & 3 & $ C_{3,1} , C_{3,2}$ &3& $D_6$&$D_{2,12,5}$& $\bbZ_2^2\times \bbZ^2$\\
\hline\hline
6 & 3 & $2 C_{2,1} $& 2& $A_4\times \bbZ_2$&G(48,30)& $\bbZ_2\times \bbZ^2$\\
6 & 7 & $2  C_{2,1} $&2& $A_4\times \bbZ_2$&$A_4\times \bbZ_4$& $\bbZ_2\times \bbZ^2$\\
6 & 5 & $ C_{5,3}$        &5& $D_5$&G(20,3)& $\bbZ_2\times \bbZ^2$\\
\hline\hline
8 &5 & $\emptyset$ & $2^2$ &$\bbZ_2\times\bbZ_4$&$D_{2,8,5}$& $\bbZ_4\times \bbZ^2$\\
8 &5 & $\emptyset$ & $2^2$ &$D_4$& $D_{2,8,3}$& $\bbZ_4\times \bbZ^2$\\
8 &5 & $\emptyset$ & $2^2$ &$\bbZ_2^3$& $\bbZ_2^2\rtimes \bbZ_4$& $\bbZ_2^3\times \bbZ^2$\\
\hline
\end{tabular}}
\caption{Mixed \qe surfaces of general type with $p_g=q=1$ }
\label{tabBigInt1}
\end{table*}

In Table \ref{tabBigInt1} we use the same notation of the previous Table \ref{tabBigInt0}; we also report
 the genus $g_{alb}$ of a general fibre of the Albanese map, and we do not report $\pi_1(S)$, which is always infinite. 
Note that there is a surface with $K_S^2=6$ and $g_{alg}=7$; to the best of our knowledge, this is the first example of a 
minimal surface of general type with $p_g=q=1$ and $g_{alb}>K_S^2$;
 we recall that this is not possible for $K^2_S \leq 3$  by the classification \cite{catCIME, CC91, CC93, CP05}.
We also note  the first example with $K^2=6$ and $g_{alb}=5$. Also the other examples with $4 \leq K^2 \leq 6$ 
may be, to the best of our knowledge, new, although other surfaces with those invariants have been already constructed 
(see \cite{somebig, Pol09, MP10, rit07, rit,  rit08}). 
 
\begin{thmC}
There exists a unique irreducible family of mixed \qe surfaces of general type with $p_g=q\geq 2$, and it has $p_g=2$ and $K^2=8$,  see Table \ref{tabBigInt2}.
\end{thmC} 

		\begin{table*}[!ht]
\centering{\footnotesize
\begin{tabular}{|c|c|c|c|c|c|}
\hline

$K^2_S$  & $\BB(X)$ & Sign. & $G^0$ & $G$  & $H_1(S,\bbZ)$  \\
\hline\hline
8 & $\emptyset$ & - & $\bbZ_2$ & $\bbZ_4$  & $\bbZ_2\times \bbZ^4$\\
\hline
\end{tabular}}
\caption{ Mixed \qe surfaces of general type with $p_g=q=2$ }
\label{tabBigInt2}
\end{table*}

The mixed \qe surfaces with $K^2_S=8\chi(S)$ are those for which the action is free; 
indeed all the examples in Tables  \ref{tabBigInt0}, \ref{tabBigInt1} and \ref{tabBigInt2} appeared in the papers cited at the beginning of this introduction. 
In particular,  the list
in \cite{Penegini} is the complete list of all the \qe surfaces with $p_g=q=2$.

An expanded version of these tables can be downloaded from:
$$\mbox{\url{http://www.science.unitn.it/~pignatel/papers/TablesMixed.pdf}}\,$$ 
 
The paper is organized as follows.
		
In Section \ref{GB} we give the algebraic recipe which, using  Riemann's Existence Theorem, constructs mixed \qe surfaces. 

In Section \ref{TBOS} we give a complete description of the analytic type of the possible singularities of $X$. We show 
moreover how to compute the number of singular points of $X$ and the analytic type of each singularity directly by the 
{\it ingredients} of the algebraic recipe above, and we give formulas for $K^2_S$, $p_g(S)$ and $q(S)$.
We think it  is worth mentioning here an unexpected consequence of those formulas (Corollary \ref{numbranch}): 
the number of branch points of the double cover $(C \times C)/G^0\rightarrow (C \times C)/G$ is even and 
bounded  above by $2(p_g(S)+1)$.

Section \ref{TAM} is devoted to the Albanese map of a mixed \qe surface with $q=1$. The main result is a formula 
to compute the genus of its general fibre.

In Section \ref{DTMM} we  show  that all 
irregular mixed \qe surfaces are minimal. In the regular case, we prove it under 
a strong assumption on the singularities of $X$ (Proposition \ref{minimalitycriterion}).

Finally, in Section \ref{FOTC}, we present our algorithm to construct all mixed quasi-\'etale surfaces 
with given values of $K^2$, $p_g$ and $q$, and prove Theorems $A$, $B$ and $C$.

{\bf Acknowledgments:} We are indebted to M. Penegini for his careful reading of a first version of this manuscript; in particular the current forms of Proposition \ref{pgTpgS} and Corollary \ref{numbranch} are due to him, we had proven only a weaker inequality.
We are grateful to I. Bauer and F. {Cata\-nese} for several inspiring conversations on group actions on the products of two curves. 
	
\section{The algebraic recipe}\label{GB}

 Throughout this paper we will denote by $C$ a Riemann surface of genus $g \geq 2$ and by $G$ a finite subgroup of $\Aut (C \times C)$ whose action is free outside a finite subset and {\it mixed},
 which means that there are elements in $G$ which exchange the two natural isotrivial fibrations of $C \times C$. 
We will denote by $G^0$ the index 2 subgroup consisting of those elements that do not exchange the factors. 

We will say that  the quotient surface  $X=(C\times C)/G$  is a \no{mixed \qe quotient}. 
We will denote by $\rho\colon S\rightarrow X$ the minimal resolution of the singularities of $X$, and we say that $S$ is a \no{mixed \qe surface}.

\begin{rem}
By \cite[Remark 2.3]{Frap11} every mixed \qe quotient is induced by a unique {\it minimal} action, which means that $G^0$ acts faithfully on both factors: 
therefore in this paper we will only consider minimal actions. If $X$ is a \no{mixed \qe surface}, then the quotient map factors as follows:
$$C\times C \stackrel{\sigma}{\longrightarrow} Y:=(C\times C)/G^0\stackrel{\pi}{\longrightarrow}X.$$
\end{rem}

\cite[Proposition 3.16]{Cat00} gives the following description
of minimal mixed actions:
\begin{thm}\label{thmix}
Let $G\subseteq \Aut(C\times C)$ be a minimal mixed action. 
Fix $\tau' \in G\setminus G^0$; it determines an element $\tau:=\tau'^2 \in G^0$
and an element $\varphi\in \Aut(G^0)$ defined by $\varphi(h):=\tau' h \tau'^{-1}$.
Then, up to a coordinate change, $G$ acts as follows:
\begin{equation}\label{action}
\begin{split}
g(x, y) &= (gx, \varphi( g)y)\\ 
\tau'g(x, y) &=(\varphi(g)y, \tau g \,x)
\end{split}\qquad for \,\,g \in G^0
\end{equation}

Conversely, for every $G^0\subseteq\Aut(C)$ and $G$ extension of degree 2 of $G^0$, fixed 
$\tau'\in G\setminus G^0$ and  $\tau$ and $\varphi$ defined as above, (\ref{action}) defines a minimal
mixed action on $C\times C$.
\end{thm}
We recall the following results:

\begin{thm}[{\cite[Theorem 2.6]{Frap11}}]\label{nonsplit}

Let $X$ be a quotient surface of mixed type provided by a minimal 
mixed action of $G$ on $C\times C$.
The quotient map $C\times C\rightarrow X$ is quasi-\'etale if and only if the exact sequence
\begin{equation}\label{ext}
1 \longrightarrow G^0 \longrightarrow G \longrightarrow \bbZ_2 \longrightarrow 1 
\end{equation}
does not split.

Moreover, if the quotient map is quasi-\'etale, then $\Sing(X)= \pi(\Sing(Y))$.
				
\end{thm}

\begin{lem}[{\cite[Lemma 2.9]{Frap11}}]\label{q=g}
Let $S\rightarrow X=(C\times C)/G$ be a mixed \qe surface.
Then $q(S)$  equals the genus of $C':=C/G^0$.
\end{lem}

The study of varieties birational to a quotient of a product of curves is strictly 
connected with the study of Galois coverings of Riemann surfaces.
Now we collect some results that allow us to shift from the geometrical setup to the 
algebraic one and viceversa.

\begin{defi}\label{gv}
Let $H$ be a finite group and let  $$g\geq 0 \quad \mbox{  and  }\quad  m_1, \ldots, m_r > 1$$
be integers.
A \no{generating vector} for $H$ of signature $(g;m_1,\ldots ,m_r)$ is a $(2g+r)$-tuple of
elements of $H$:
$$V:=(d_1,e_1,\ldots, d_g,e_g;h_1, \ldots, h_r)$$
\sts $V$ generates $H$, $\prod_{i=1}^g[d_i,e_i]\cdot h_1\cdot h_2\cdots h_r=1$ and there 
exists a permutation $\sigma \in \mathfrak{S}_r$ such that $\ord(h_i)=m_{\sigma(i)}$ for $i=1,\ldots, r$.
In this case, we also say that $H$ is \no{$(g;m_1,\ldots,m_r)$-generated}. 
\end{defi}

By Riemann's Existence Theorem (see \cite{Survey}), any curve $C$ of genus $g$ together with an action of
a finite group $H$ on it, such that $C/H$ is a curve $C'$ of genus $g'$, is determined (modulo automorphisms)
by the following data:
\begin{enumerate}
	\item the branch point set $\{p_1, \ldots, p_r\}\subset C'$;
	\item loops $\alpha_1,\ldots,\alpha_{g'},\beta_1,\ldots,\beta_{g'},\gamma_1,\ldots,\gamma_r\in \pi_1(C'\setminus\{p_1, \ldots, p_r\})$, where $\{\alpha_i, \beta_i\}_i$ generates $\pi_1(C')$, each $\gamma_i$ is a simple geometric loop around $p_i$ and 
	 $\prod_{i=1}^{g'}[\alpha_i,\beta_i]\cdot\gamma_1\cdot\ldots\cdot\gamma_r=1\,;$
	\item a \no{generating vector} for $H$ of signature $(g';m_1,\ldots ,m_r)$
	with the property that \no{Hurwitz's formula} holds:
	\begin{equation}
	2g-2=|H|\bigg(2g'-2+\sum_{i=1}^r\frac{m_i-1}{m_i}\bigg)\,.
	\end{equation}
	\end{enumerate}

\begin{rem}\label{algdata} Analogously, a mixed \qe quotient $ X=(C\times C)/G$ determines  a finite group $G^0$,  a degree $2$ extension 
$1 \rightarrow G^0 \rightarrow G \rightarrow \bbZ_2 \rightarrow 1 $, the curve  $C'= C/G^0$,  a set of points $\{p_1,\ldots, p_r\}\subset C'$, 
and, for every choice of $\alpha_i, \beta_j, \gamma_k \in \pi_1(C'\setminus\{p_1, \ldots, p_r\})$ as in (2),  a  generating vector $V$ for $G^0$.

Conversely, the following algebraic data:
\begin{itemize}
 	 \item a finite group $G^0$;          
            \item a curve  $C'$;
	  \item points $p_1,\ldots, p_r\in C'$, and $\alpha_i, \beta_j, \gamma_k \in \pi_1(C'\setminus\{p_1, \ldots, p_r\})$ as in (2);
             \item integers $m_1,\ldots,m_r >1$;
\item a  generating vector $V$ for $G^0$ of signature $(g(C');m_1,\ldots,m_r)$;
\item a degree $2$ extension $1 \rightarrow G^0 \rightarrow G \rightarrow \bbZ_2 \rightarrow 1 $
	which does not split;
\end{itemize}

\noindent 
give a uniquely determined mixed \qe quotient. Indeed
by Riemann's  Existence Theorem the first $5$ data give the Galois cover $c\colon C \rightarrow C/G^0\cong C'$ branched over $\{p_1, \ldots, p_r \}$. 
The last datum determines, by Theorem \ref{thmix}, a minimal mixed action on $C\times C$ and 
by Theorem \ref{nonsplit} the action is free outside a finite set of points.

 If $V:=(d_1,e_1,\ldots, d_g,e_g;h_1, \ldots, h_r)$ we will denote by $K_i$ the cyclic subgroup of $G^0$ generated by $h_i$.

\end{rem}

\section{The singularities of a mixed \qe quotient}\label{TBOS}
This section is devoted to the study of the singularities of a mixed \qe quotient $X=(C\times C)/G$. We will need to consider the intermediate quotient $Y=(C \times C)/G^0$, and the two isotrivial fibrations $\alpha_i \colon Y \rightarrow C'=C/G^0$ induced by the projections of $C \times C$ on the two factors.

The double cover $\pi \colon Y \rightarrow X$ determines an involution $\iota \colon Y \rightarrow Y$ such that
$X=Y/\iota$. 
By the last statement of Theorem \ref{nonsplit}, the fixed points of $\iota$ are singularities of $Y$, hence $\iota$ splits the singularities of $X$ in two classes: 
the singularities not in the branch locus of $\pi$ (analytically isomorphic to each of its preimages in $Y$), and the images of the fixed points of $\iota$. We need then to consider the singularities of $Y$ and the action of $\iota$ on them.

$Y$ is a product-quotient surface, whose singularities are now well understood (see \cite{BP10,MP10,Pol10}). 
They are {\it cyclic quotient singularities}, 
isomorphic to the quotient $\bbC^2/\langle \sigma \rangle$, where $\sigma$ is the diagonal linear automorphism with eigenvalues $\exp({\frac{2\pi i }{n}})$
and $\exp({\frac{2\pi i a }{n}})$
with $n>a>0$ and $\gcd(a,n)=1$. We will say that this is a \no{singularity of type $C_{n,a}$}. 
Two singularities of respective types $C_{n,a}$ and $ C_{n',a'}$ are locally analytically isomorphic 
if and only if $n=n'$ and either $a=a'$ or $aa'\equiv 1\mod n$. We read from \cite{BP10} how to determine the singular points of $Y$ and their respective $n$ and $a$.

\begin{prop}[{\cite[Propositions 1.16 and 1.18]{BP10}}]\label{s-h}
Let $X$ be a mixed \qe quotient given by data as in Remark \ref{algdata}, let $Y=(C\times C)/G^0$ be the intermediate product-quotient surface,
 and consider the induced map $Q=(\alpha_1,\alpha_2)\colon Y \rightarrow C' \times C'$. The singular points of $Y$ are the points $y=\sigma(u,v)$ \sts 
$$\Stab_{G^0}(u)\cap \varphi^{-1} (\Stab_{G^0}(v)) \neq \{1\}\,,$$ where $\varphi$ is the automorphism of $G^0$ in Theorem \ref{thmix}. 
In particular, if $y\in \Sing(Y)$ then $Q(y)=(p_i,p_j)$ for some $i,\,j$. Now fix $i,\,j  \in \{1,\ldots,r\}$, then
\begin{itemize}
	\item[i)] there is a $G^0$-equivariant bijection $(Q \circ \sigma )^{-1}(p_i,p_j)\rightarrow G^0/K_i \times G^0/K_j$,
	where the action on the target is $g(aK_i,bK_j)=(gaK_i,\varphi(g)bK_j)$;
	\item[ii)] there is a $K_i$-equivariant bijection between the orbits of the above
	$G^0$-action on $G^0/K_i \times G^0/K_j$ with the orbits of the $K_i$-action on $\{\ov{1}\}\times G^0/K_j$.
\item[iii)] An element $[g]\in \{\ov{1}\}\times G^0/K_j$
corresponds to a point of type $C_{n,a}$ on $Y$, where $n=|K_i\cap \varphi^{-1}(gK_jg^{-1})|$, and 
$a$ is given as follows:
let $\delta_i$ be the minimal positive integer such that
there exists $1\leq \gamma_j \leq \ord(h_j)$ with $h_i^{\delta_i}= g\varphi^{-1}(h_j^{\gamma_j})g^{-1}$.
Then $a=\dfrac{n\gamma_j}{\ord(h_j)}$.
\end{itemize}
\end{prop}
By Proposition \ref{s-h} we can compute the singularities of $Y$ from the algebraic data of Remark \ref{algdata}.
In order to compute the basket of singularities of $X$,
 we first need to know which of them are
ramification points for $\pi$.

\begin{lem}[{\cite[Proposition 3.8]{Frap11}}]
Let $y\in Y$ be a fixed point for $\iota$.
Then $Q(y)=(p_i,p_i)$ for some $i$. In other words, $Q$ maps all fixed points of $\iota$ to the diagonal of $C' \times C'$.
\end{lem}

\begin{prop}\label{fixi}
An element $[g]\in  \{\ov{1}\}\times G^0/K_i $ corresponds
to a fixed point
for $\iota$ if and only if
there exists an element $h \in G^0$ \st:
$$\left\{
\begin{array}{l}
\varphi(h)\tau h \in K_i\\
\varphi(h) g \in K_i
\end{array}
\right.$$
\end{prop}

\begin{proof}
The point $(K_i, gK_i)$ corresponding to $[g]$ is a ramification point 
for $\pi$ if and only if there exists an element $\tau' h \in G \setminus G^0$ \sts
$(K_i,gK_i)=\tau'h(K_i,gK_i)=(\varphi(h)gK_i, \tau h K_i)$, that is 
$$\left\{
\begin{array}{l}
\varphi(h)gK_i= K_i\\
 g K_i =\tau h K_i
\end{array}
\right.\quad
\Longleftrightarrow\quad
\left\{
\begin{array}{l}
\varphi(h) g K_i = K_i\\
\varphi(h)\tau h K_i= (\tau'h)^2K_i=K_i
\end{array}
\right.
$$
\end{proof}

We study now the action of $\iota$ on a neighbourhood of a singular point of $Y$.
We denote by $\lambda \colon T \rightarrow Y$ the minimal resolution of the singularities of $Y$. 
The exceptional divisor $E$ of the minimal resolution 
of a cyclic quotient singularities of type $C_{n,a}$ is a \no{Hirzebruch-Jung string},
that is $E= \sum_{i=1}^l E_i$, where the $E_i$ are smooth
rational curves with $E_i.E_{i+1}=1$, $E_i.E_j=0$ for $|i-j|\geq 2$, and $E_i^2=-b_i$ where the $b_i$ are the coefficients 
of the continued fraction of $\frac{n}{a}$:
$$\frac{n}{a}=b_1-\frac{1}{b_2-\frac{1}{b_3-\ldots}}=: [b_1,\ldots, b_l]\,.$$

\begin{rem} $a\cdot a'\equiv 1\mod n$ if and only if the continued fraction of $\frac{n}{a'}$ is $ [b_l,\ldots, b_1]$.
\end{rem}

We need the following
\begin{lem}\label{liftinv}
The involution $\iota$ on $Y$ lifts to a morphism $\mu \colon T\rightarrow T$. 
\end{lem}

\begin{proof}
Consider $\mu:= \lambda^{-1} \circ \iota \circ\lambda \colon T \dashrightarrow T$.
Let $\Gamma\subset T\times T$ be the graph of $\mu$;
let $f_1, f_2 \colon \Gamma \rightarrow T$  be the projections on the factors.

If $\mu$ is not defined at a point $p\in T$, then  $\Gamma$ contains a $(-1)$-curve $C$ contracted to $p$
by $f_1$.  $D:=f_2(C)\subset T$ is a curve contracted to $\iota(\lambda(p))$ by $\lambda$, so a component of a H-J string: in particular $D^2\leq-2$.
On the other hand, since $f_2$ is a birational morphism, $D^2\geq C^2=-1$, a contradiction.
\end{proof}

To study the action of $\mu$ on the H-J strings, we will need the following  

\begin{prop}[{see \cite[Theorem 2.1]{ser96}}]\label{invfib}
Let $y\in Y$ be a singular point of type $C_{n,a}$, and consider the two fibres 
$F_1:=\alpha_1^{*}(\alpha_1(y))$ and $F_2:=\alpha_2^{*}(\alpha_2(y))$ taken with the reduced structure.
Let $\tilde{F_i}:=\lambda^{-1}_*(F_i)$ be the strict transforms of $F_i$ ($i=1,\,2$) and 
let $E$ be the exceptional divisor of $y$.

Then $\tilde{F_1}$ intersects one of the extremal curves of $E$, say $E_1$, while
$\tilde{F_2}$ intersects the other extremal curve, say $E_l$.
\end{prop}

Proposition \ref{invfib} motivates the following

\begin{defi}\label{wrt}
Let $\alpha\colon Y\rightarrow C'$ be one of the two natural fibrations.
Let $y\in \Sing(Y)$ be a point of type $C_{n,a}$.
Let $E:=\sum_{i=1}^l	E_i$ be the exceptional divisor over $y$, where the $E_i$ are rational curves ordered so that
$E_i^2=-b_i$, $E_i.E_{i+1}=1$.
Let $\tilde{F}$ be the strict transform in $T$ of  the fibre $F=\alpha^{*}(\alpha(y))$ taken with
the reduced structure.\\
We say that $y$ is  \no{of type $C_{n,a}$ with respect to $\alpha$} if
$\tilde{F}$ intersects $E_1$.
\end{defi}

\begin{rem}
\noindent If $y$ is  of type $C_{n,a}$ with respect to $\alpha_1$ then $y$ is  of type $C_{n,a'}$ with respect to $\alpha_2$, with $a\cdot a'\cong 1 \mod n$.
\end{rem}

\begin{lem}\label{singinv}
If $y$ is a point of type $C_{n,a}$ with respect to $\alpha_1$,
then $\iota(y)$ is a point of type $C_{n,a'}$ with respect to $\alpha_1$, with $a\cdot a'\cong 1 \mod n$.
\end{lem}

\begin{proof}
Since $\iota$ is an isomorphism, $y$ and $z:=\iota(y)$ have the same analytic type, so $z$  is either of type  $C_{n,a}$  or of type $C_{n,a'}$ with respect to $\alpha_1$.

Let $Y_i$, resp. $Z_i$ be the fibre of $\alpha_i$ containing $y$, resp. $z$,
 all of them taken with the reduced structure and let
  $\tilde{Y_i}:=\lambda^{-1}_*(Y_i)$ and  $\tilde{Z_i}:=\lambda^{-1}_*(Z_i)$ ($i=1,2$) be their strict transforms in $T$.
Note that $\iota$ is the map induced on $Y$ by the action on $C \times C$ of any $\tau' \in G \setminus G^0$. Since $\tau'$ exchanges the two factors, 
then $\iota(Y_1)=Z_2$ and therefore $\mu(\tilde{Y_1})=\tilde{Z_2}$.

Let $E=\sum_{i=1}^l E_i$ resp. $E'=\sum_{i=1}^l E'_i$ be the exceptional divisor of $y$ resp. $z$, with the $E_i$ resp. $E'_i$ ordered as in Definition \ref{wrt} for $\alpha=\alpha_1$. By assumption $\tilde{Y_1}$ intersects $E_1$, $\tilde{Z_1}$ intersects $E'_1$, $\tilde{Z_2}$ intersects $E'_l$.

Since $\mu(\tilde{Y_1})=\tilde{Z_2}$, then $\mu(E_1)=E'_l$. It follows that $z$ is of type $C_{n,a}$  with respect to $\alpha_2$ and of type $C_{n,a'}$ with respect to $\alpha_1$.
\end{proof}

We give now a full description of the singular points of $X$ arising from fixed points of $\iota$.
\begin{prop}\label{quotbyinv}
Let $X=(C\times C)/G$ be a mixed \qe quotient and let $y\in Y$ be a fixed point of $\iota$. Then $y$
is a singularity of type $C_{n,a}$ with $a^2\equiv 1\mod n$; so $a=a'$and the continued fraction $\dfrac{n}{a}=[b_1,\ldots,b_l]$ is palindromic: 
$b_i=b_{l+1-i}\ \forall i$.
 
Moreover
\begin{itemize}
\item[(i)] $n$ is even;
	\item[(ii)]  $l$ is odd: $l=2m+1$ and $b_{m+1}$ is even;
	\item[(iii)] the exceptional divisor of the minimal resolution of the singular point $\pi(y)$ is a tree of $m+3$ smooth rational curves with decorated dual graph:
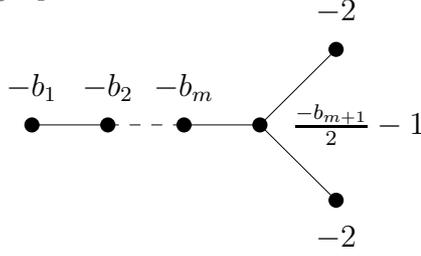
\begin{center}
\scalebox{1} 
{
\begin{pspicture}(0,-1.5)(5,1.5)
\psdots[dotsize=0.2](0,0)
\psdots[dotsize=0.2](1,0)
\psdots[dotsize=0.2](2,0)
\psdots[dotsize=0.2](3,0)
\psdots[dotsize=0.2](4,1)
\psdots[dotsize=0.2](4,-1)
\psline[linewidth=0.01cm](0,0)(1,0)
\psline[linewidth=0.01cm,linestyle=dashed,dash=0.16cm 0.16cm](1,0)(2,0)
\psline[linewidth=0.01cm](2,0)(3,0)
\psline[linewidth=0.01cm](3,0)(4,1)
\psline[linewidth=0.01cm](3,0)(4,-1)
\usefont{T1}{ptm}{m}{n}
\rput(0,0.5){$-b_1$}
\rput(1,0.5){$-b_2$}
\rput(2,0.5){$-b_m$}
\rput(4.3,0){$\frac{-b_{m+1}}{2}-1$}
\rput(4,1.5){$-2$}
\rput(4,-1.5){$-2$}
\end{pspicture} 
}

\end{center}
\end{itemize}
\end{prop}

\begin{proof}
 By Lemma \ref{singinv} $y$ is of type $C_{n,a}$ with respect to both $\alpha_j$, so $a=a'$ and $b_i=b_{l+1-i}$.
 More precisely, the proof of Lemma \ref{singinv} shows that, if $E=\sum_1^l E_i$ is the H-J string of $y$, $\mu(E_i)=E_{l+1-i}$.

(i) If $y=\sigma(u,v)$, $|\Stab_{G^0}(u,v)|=n$ and $|\Stab_G(u,v)|=2n$.
If $n$ is odd, then by Sylow's theorem there exists an element $g$ of order 2 in $\Stab_G(u,v) \setminus \Stab_{G^0}(u,v)$,  splitting
the exact sequence (\ref{ext}), a contradiction.

(ii)
 Let $D= \sum_{i=1}^lD_i:=\lambda^{-1}(y)$, and
assume that $l=2m$ is even. The involution $\mu$ exchanges $D_i$ with $D_{l+1-i}$,
hence $p=D_m\cap D_{m+1}$ is the unique point of $D$ fixed by $\mu$.
$\ud\mu_p$ exchanges the directions of the tangent spaces of $D_m$ and $D_{m+1}$ and therefore it is not a multiple of the identity.
Since it is an involution, then up to a linear coordinate change
$\ud\mu_p=\left(\begin{array}{cc}-1&0\\0&1\end{array}\right)$,
which implies that the fixed locus of $\mu$ contains a curve through $p$ in the direction of the eigenspace with eigenvalue $1$, a contradiction. We delay the proof that $b_{m+1}$ is even.

(iii) By part (ii), $l=2m+1$ and all fixed points of $\mu$ in $D$ belong to $D_{m+1}=\mu(D_{m+1})$.
The restriction of $\mu$ to $D_{m+1}$ is an involution and therefore
by Hurwitz's formula it fixes exactly two points $p_1$ and $p_2$,
that are distinct from the points of intersection  of $D_{m+1}$ with $D_m$ or $D_{m+2}$.

Let $V$ be a small $\mu$-invariant open set of $T$ containing $D$ and not intersecting any other exceptional divisor, so that $\lambda(V)$ is an open set of $Y$ 
containing only one singular point: $y$. Let $\epsilon\colon V'\rightarrow V$ be the blow-up in $p_1$ and $p_2$,
we denote by $D'_i$ the strict transform of $D_i$ and by $A_1$ and $A_2$ the two
$(-1)$-exceptional curves. The involution $\mu$ lifts to an involution $\mu'$ on $V'$ whose fixed locus is the smooth curve $A_1 \cup A_2$.

Then $V'/\mu'$ is smooth, and therefore is a resolution of the singular point $\pi(x)$ whose exceptional divisor is $D/\mu$.
The computation of the dual graph of  $D/\mu$ is a standard computation that we leave to the reader. We notice that there is no curve with self-intersection $-1$, so the resolution is the minimal resolution. Moreover there is a curve of self-intersection $-(1+b_{m+1}/2)$, showing that $b_{m+1}$ is even.
\end{proof}

It follows that the analytic type of a singularity on $X$ only depends on its preimage on $Y$. Indeed, these quotient singularities can be described as follows:

\begin{prop}\label{gD}
Let $X=(C\times C)/G$ be a mixed \qe quotient and let $y\in \Sing(Y)$ be a point of type $C_{n,a}$ with 
$\dfrac{n}{a}=[b_1,\ldots, b_m, 2b,b_m,\ldots, b_1]$.
Let 
$$\dfrac{p}{q}:=[b_1,\ldots, b_m]\,,\qquad  \mbox{ and } \quad
\xi:=bp-q\,.$$

\noindent If $y$ is a ramification point for $\pi$,  then $x:=\pi(y)$ is a quotient singularity isomorphic to $\bbC^2/H$
with:
\begin{itemize}
	\item if $\xi=0$ (\ie $p=0$), then 
	$$H=\left\langle\left(
\begin{array}{cc}
\epsilon &0\\
0 & \epsilon^{n+1}
\end{array}
\right)
\right\rangle\,, \quad \mbox{ with } \epsilon= e^{\frac{2\pi i}{2n}},$$

\item if $\xi \neq 0$ and odd, then
	$$H=\left\langle
	\left(
\begin{array}{cc}
\eta &0\\
0 & \eta
\end{array}
\right)
\,,
	\left(
\begin{array}{cc}
\omega &0\\
0 & \omega^{-1}
\end{array}
\right)
\,,
	\left(
\begin{array}{cc}
0 &1\\
-1 &0
\end{array}
\right)
\right\rangle\,,
\mbox{ with } \eta= e^{\frac{2\pi i}{2\xi}}\,, \omega= e^{\frac{2\pi i}{2p}},$$
 
\item if $\xi \neq 0$ and even, then
	$$H=\left\langle
	\left(
\begin{array}{cc}
0&\zeta\\
- \zeta& 0
\end{array}
\right)
\,,
	\left(
\begin{array}{cc}
\omega &0\\
0 & \omega^{-1}
\end{array}
\right)
\right\rangle\,,
 \, \mbox{ with } \zeta= e^{\frac{2\pi i}{4\xi}} \mbox{ and } \omega= e^{\frac{2\pi i}{2p}}.$$

\end{itemize}
\end{prop}

\begin{proof}
The statement follows immediately from the classification of finite subgroups of $\GL(2,\bbC)$
without quasi-reflections, see \cite[Satz 2.11]{Brieskorn68} or 
\cite[Theorem 4.6.20]{Matsuki}.
\end{proof}

\begin{defi}\label{defDna}
We say that a singular point $x$ as in Proposition \ref{gD} is a singular point of type $D_{n,a}$.
\end{defi}

\begin{rem}\quad

\begin{enumerate}
	\item  A singular point of type $D_{n,a}$ is a Rational Double Point if and only if $a=n-1$, in which case we have a Rational Double Point 
	of type $D_{\frac{n}{2}+2}$ (if $n=2$, this is more commonly known as $A_3$).
	\item  A singular point of type $D_{n,a}$ is a cyclic quotient singularity if and only if $a=1$.
	More precisely singularities of type $D_{n,1}$  are isomorphic to singularities of type $C_{2n,n+1}$.
	We will distinguish between them keeping track of the branching locus of $\pi$. 
\item  We noted that a point of type $C_{n,a}$ is also a point
of type $C_{n,a'}$ with $a'=a^{-1}$ in $\bbZ_n$.
We consider these different representations as equal and usually we do not distinguish between them. 
\end{enumerate}
\end{rem}

In the following the term \no{multiset} will be used in the sense of MAGMA \cite{MAGMA}. So a multiset is a set whose elements have a {\it multiplicity}: 
a positive integer; the cardinality of a multiset takes into account the multiplicity of its elements.

\begin{defi}
Let $Y$ be an unmixed surface. Then we define the \no{basket of singularities of
$Y$} to be the multiset
$$\BB(Y):=\big\{\lambda\times C_{n,a}\, :\, Y \mbox{ has exactly } \lambda \mbox{ singularities of type }
C_{n,a} \big\}\,.$$
\end{defi}

Let $X=(C\times C)/G$ be a mixed \qe quotient.
We define the following two multisets:
\begin{multline*}
\BB_C:=\big\{\eta\times C_{n,a}\, :\,X \mbox{ has exactly } \eta \mbox{ singularities of type } C_{n,a}\\
\mbox{ not in the branch locus of $\pi$ } \big\}\,.
\end{multline*}
\begin{multline*}
\BB_D:=\big\{\zeta \times D_{m,b}\, :\, X \mbox{ has exactly } \zeta \mbox{ singularities of type }D_{m,b}\\
\mbox{ in the branch locus of $\pi$ } \big\}\,.
\end{multline*}

\begin{defi}\label{defbasket}
The \no{basket of singularities of $X$} is the multiset
$$\BB(X)=\BB_C\cup \BB_D\,. $$
\end{defi}

The following is an useful constraint on the basket of singularities.

\begin{prop}\label{sumofsing}
Let $X=(C\times C)/G$ be a mixed \qe quotient.
Let $\BB(X)=\BB_C\cup \BB_D $ be the basket of singularities of $X$ with  $\BB_C:=\{\eta_i\times C_{n_i,a_i}\}_i$ 
and $\BB_D:=\{\zeta_j \times D_{m_j,b_j}\}_j$.
Then
$$\sum_i \eta_i \frac{a_i+a'_i}{n_i}+\sum_j\zeta_j \frac{b_j}{m_j} \in \bbZ\,.$$
\end{prop}

\begin{proof}
If $x\in X$ is a singular point of type $C_{n,a}$, then by Lemma \ref{singinv} $\pi^{-1}(x)$ is given by two singular points, one of type $C_{n,a}$ with 
respect to $\alpha_1$ and the other of type $C_{n,a'}$ with respect to $\alpha_1$.
If $x\in X$ is a singular point of type $D_{m,b}$, then $\pi^{-1}(x)$ is given by a unique singular point of type $C_{m,b}$
with respect to $\alpha_1$.
The result now follows directly from \cite[Proposition 2.8]{Pol10}.
\end{proof}

\begin{defi}
Let $x$ be a singular point of type $C_{n,a}$ with $\frac{n}{a}:=[b_1,\ldots, b_l]$. We define the following nonnegative rational numbers
\begin{itemize}
	\item[i)] $k_x=k(C_{n,a}):=-2+\dfrac{2+a+a'}{n}+\sum_{i=1}^l(b_i-2)$;
	\item[ii)] $e_x=e(C_{n,a}):=l+1-\dfrac{1}{n}\geq 0$;
\end{itemize}

\noindent Let $x$ be a singular point of type $D_{n,a}$ with $\frac{n}{a}:=[b_1,\ldots, b_m,2b,b_m,\ldots, b_1]$. We define the analogous nonnegative rational numbers
\begin{itemize}
	\item[i)] $k_x=k(D_{n,a}):= \frac{k(C_{n,a})}{2}=-2+\frac{a+1}{n}+\sum_{i=1}^m(b_i-2)+b$;
	\item[ii)] $e_x=e(D_{n,a}):=\frac{e(C_{n,a})}{2}+3=m+4-\dfrac{1}{2n}$;
\end{itemize}
In both cases we set $B_x:=2e_x+k_x$. Note that $B(D_{n,a})=\frac{B(C_{n,a})}{2}+6$.

\noindent Let $\BB$ be a basket of singularities. We use the following notation:
$$k(\BB)=\sum_{x\in\BB}k_x,\,\quad e(\BB)=\sum_{x\in\BB}e_x,\,\quad B(\BB)=\sum_{x\in\BB}B_x\,.$$
\end{defi}

These correction terms determine the invariants of $S$ as follows:

\begin{prop}\label{Ksquare} 
Let $\rho\colon S\rightarrow X=(C\times C)/G$ be  a mixed \qe surface, and let $\BB$ be the basket of singularities of $X$. Then
\begin{equation}\label{Kmix}
K^2_S=\dfrac{8(g-1)^2}{|G|}-k(\BB)\,;
\end{equation}
\begin{equation}\label{emix}
e(S)=\dfrac{4(g-1)^2}{|G|}+e(\BB)\,.
\end{equation}
\end{prop}

\begin{proof}
Since the quotient map $C\times C\rightarrow X$ is quasi-\'etale, we get
$$K_X^2=\dfrac{K^2_Y}{2}=\dfrac{8(g-1)^2}{|G|}\,.$$

\noindent Let $\BB=\BB_C\cup \BB_D=\{\eta_i\times C_{n_i,a_i}\}_i \cup\{\zeta_j \times D_{n_j,a_j}\}_j$, then
the basket of singularities of $Y
$ is $\BB(Y)=\{2\eta_i\times C_{n_i,a_i}\}_i\cup\{\zeta_j \times C_{n_j,a_j} \}_j\,,$
hence by definition $k(\BB(Y))=2k(\BB)$. By \cite[Proposition 2.6]{BCGP08}, we get
$$K^2_T=\dfrac{8(g-1)^2}{|G^0|}-k(\BB(Y))\,.$$
Let $\epsilon \colon T'\rightarrow T$ be the blow-up of $T$ in the $2d$ ($d=|\BB_D|$) points fixed by $\mu$:
\begin{equation}\label{eq1}
K_{T'}^2=K_T^2-2d=K^2_Y-k(\BB(Y))-2d=2(K_X^2-k(\BB)-d)\,.
\end{equation}

\noindent By the proof of Proposition \ref{quotbyinv} we have a double cover   $\tilde{\pi}\colon T'\rightarrow S$ branched over 
$F:=F_1+\ldots+ F_{2d}$, where the $F_i$ are smooth rational curves with $F_i^2=-2$ and
and $F_i.F_j=0$ if $i\neq j$.
Then numerically (\cite[pages 13-14]{CD}) $K_{T'}$ equals $\tilde{\pi}^*(K_S+F/2)$ 
and, since $K_S.F=0$, it follows:
\begin{equation}\label{eq2}
K_{T'}^2=2\bigg(K_S+\frac{F}{2}\bigg)^2=2\bigg(K_S^2+\frac{-4d}{4}\bigg)=2(K_S^2-d)\,.
\end{equation}
From equations (\ref{eq1}) and (\ref{eq2}), we get:
$$K_S^2=K_X^2-k(\BB)=\dfrac{8(g-1)^2}{|G|}-k(\BB)\,.$$

Let $X^0:=X\setminus \Sing(X)$ be the smooth locus of $X$;
arguing as in \cite{BCGP08} we get:

$$e(S)=e(X^0)+\sum_{x\in \BB_C}(l_x+1)+\sum_{x\in \BB_D}(m_x+4)\,$$
and
 
\begin{equation*}
e(X^0)= \frac{e(C\times C)}{|G|}-\sum_{x \in \BB_C} \frac{1}{n_x}-\sum_{x\in \BB_D}\frac{1}{2n_x}
\end{equation*}

\noindent It follows that
\begin{equation*}
e(S)= \frac{4(g-1)^2}{|G|}+e(\BB)
\end{equation*}
\end{proof}

Using Noether's formula and Proposition \ref{Ksquare} we get:

\begin{cor}\label{chiKB}
Let $\rho\colon S\rightarrow X=(C\times C)/G$ be  a mixed \qe surface,  and let $\BB$ be the basket of singularities of $X$. Then
$$K^2_S=8\chi(S)-\frac{1}{3}B(\BB)\,.$$
\end{cor}

We conclude this section by showing a very strong restriction on the cardinality of $\BB_D$.

\begin{cor}\label{numbranch}
Let $\rho\colon S\rightarrow X=(C\times C)/G$ be a mixed \qe surface.
The cardinality $d$ of $\BB_D$  is even and
$$\frac{d}{2}\leq p_g(S)+1\,.$$
\end{cor}
Corollary \ref{numbranch} follows from the next proposition since the singular points of $X$ of type $D_{n,a}$ are the branch points of $\pi$.

\begin{prop}\label{pgTpgS}
Let $\rho\colon S\rightarrow X=(C\times C)/G$ be a mixed \qe surface and let $\lambda\colon T\rightarrow Y$ be the minimal resolution
 of the singularities of $Y$. Let $d$ be the number of fixed points for $\iota$, then
$$p_g(S) \leq p_g(T)=2p_g(S)+1-\frac{d}{2}\,.$$
\end{prop}

\begin{proof}
Let $\epsilon \colon T'\rightarrow T$ be the blow-up of $T$ in the $2d$ points fixed by $\mu$; we have a double cover
 $\tilde{\pi}\colon T'\rightarrow S$  branched along $2d$ smooth pairwise disjoint rational curves. Pulling back the forms on $S$ to forms on $T'$ we note that
 $p_g(S) \leq p_g(T')=p_g(T)$. Moreover 
$e(T')=2e(S)-4d$, $e(T)=e(T')-2d=2e(S)-6d$ and $K^2_T=2K^2_S$, by (\ref{eq1}) and (\ref{eq2}).
By Noether's formula:
\begin{equation*}
\chi(\OO_T)=\frac{1}{12}(K^2_T+e(T))=\frac{1}{12}(2K^2_S+2e(S)-6d)=2\chi(\OO_S)-\frac{d}{2}
\end{equation*}

\noindent Since $T\rightarrow Y$ is a product-quotient surface, $q(T)=2g(C/G^0)=2q(S)$ and 
\begin{equation*}
p_g(T)=2+2p_g(S)-2q(S)-\frac{d}{2} + q(T)-1=2p_g(S)+1-\frac{d}{2}.
\end{equation*}
 \end{proof}

\section{The Albanese fibre of a mixed \qe surface with irregularity 1}\label{TAM}

The Albanese map of a surface of general type $S$ with irregularity $1$ is a fibration onto the elliptic curve 
$\alb(S)$. The genus $g_{alb}$ of the general Albanese fibre is a deformation invariant, which is very important from the point of view of the geography of  surfaces of general type. In this section we show how to compute $g_{alb}$ for mixed \qe surfaces.

Let $S\stackrel{\rho}{\rightarrow} X=(C\times C)/G$ be a mixed \qe surface with $q(S)=1$.
By Lemma \ref{q=g}, $C'=C/G^0$ is an elliptic curve, so in this section we will set $E:=C'$. 
We have the following commutative diagram:	
\begin{equation}\label{albdiag}
\xymatrix{
                          	    &C\times C \ar[r]^Q\ar[d]_\varsigma                   & E\times E\ar[d]^\epsilon\\
S\ar[r]^\rho \ar[rrd]^\alpha\ar[rd]_f  & X \ar[r]                                &E^{(2)}\ar[d]^{\tilde{\alpha}}\\
																& \alb(S) \ar[r]_\psi& E
}
\end{equation}
where $\tilde{\alpha}$ is  the  Abel-Jacobi map.
By the properties of the Albanese torus (see \cite[Proposition I.13.9]{BHPV}), 
 the Stein factorization of $\alpha$ is given by the Albanese map $f\colon S\rightarrow \alb(S)$ and a (unique)
homomorphism $\psi\colon \alb(S) \rightarrow E$. 

The Galois cover $c\colon C\rightarrow E$ has branching set $B:=\{p_1,\ldots, p_r\}$; 
up to translation we may assume that the neutral element $0$ of $E$  is not in $B$, and that 
$-p_i\notin B$ for each $i\in\{1,\ldots, r\}$.

Let $E':=\epsilon^*(\tilde{\alpha}^*(0))=\{(x,-x) \mid x\in E\}\cong E$, consider
$F^*:=Q^*(E')$ and let $F:=\alpha^*(0)$. Note that $\rho(F)=\varsigma(F^*)$.
Our assumption $-p_i\notin B$ ensures that $F^*$ and $F$ are smooth, and the arithmetic genus of $F$ can be easily computed by 
Hurwitz's Formula, see equation (\ref{Halb}) below.  
$F$ is the disjoint union of $\deg \psi$ fibres of the Albanese map, so to compute $g_{alb}$ we need to compute $\deg \psi$ first.

We will need the points $q_i:=(p_i, -p_i)$ and $q'_i:=(-p_i,p_i)$ of $E'$; we set $B':=\{q_i,q_i'\}_i$. 
We note that $0'=(0,0)\in E'\setminus B'$.

\noindent By Remark \ref{algdata}, given suitable loops $\alpha,\beta,\gamma_1,\ldots,\gamma_r\in \pi_1(E\setminus B,0)$,  
the cover $c \colon C \rightarrow E$ is determined by a generating vector $(a,b;h_1, \ldots, h_r)$ of $G^0$, representing 
the monodromy map $\mu \colon \pi_1(E\setminus B,0) \rightarrow G^0$ of $c$.

Since $Q=c \times c$, the monodromy map of the $G^0 \times G^0$-cover $Q$ is given by two copies of $\mu$. 
$Q$ induces by restriction the $G^0 \times G^0$-cover $F^*\rightarrow E'$, whose branching locus is $B'$. 
To describe its monodromy map we choose generators
$\delta,\theta,\gamma'_1,\ldots,\gamma'_r,\gamma''_1,\ldots,\gamma''_r \in \pi_1(E'\setminus B',0')$ 
as follows:
\begin{itemize}
\item $\delta=(\alpha, -\alpha)$  
\item $\theta=(\beta, -\beta)$  
\item $\gamma'_i=(\gamma_i,-\gamma_i)$ are geometric loops around $q_i$  
\item $\gamma''_i=(-\gamma_i,\gamma_i)$ are geometric loops around $q'_i$  
\end{itemize} 
Please note that we need some care in the choice of the loops $\alpha,\beta,\gamma_i$ to ensure that $\delta,\theta,\gamma'_i$ and $\gamma''_i$  do not 
meet $B'$. 

Moreover, the class of  $\delta$,  $\theta$, $\gamma'_i$ and $\gamma''_i$ in $\pi_1(E'\setminus B',0')$ depends on the choice of the loops 
$\alpha,\beta,\gamma_i$ and not only on their class in $\pi_1(E\setminus B,0)$. Anyway, the classes of $\delta,\theta,\gamma'_i$ and $\gamma''_i$
generate $\pi_1(E'\setminus B',0')$ and the monodromy map of
$Q_{|F^*}\colon F^* \rightarrow E'$ is the unique homomorphism $\pi_1(E'\setminus B', 0')\stackrel{\mu'}{\longrightarrow} G^0\times G^0$ such that 
\begin{equation}\label{mon}
\begin{array}{ccc}
\delta &\stackrel{\mu'}{\longmapsto}& (a,a^{-1})\,,\\
\gamma'_i &\stackrel{\mu'}{\longmapsto}& (h_i, 1)\,, \\
\end{array} \qquad \qquad
\begin{array}{ccc}
\theta &\stackrel{\mu'}{\longmapsto}& (b,b^{-1})\,,\\
\gamma''_i &\stackrel{\mu'}{\longmapsto}& (1,h_i)\,.\\
\end{array}
\end{equation}


\begin{rem} \label{defM}
1) We note that the index of $\im(\mu')$ in $G^0\times G^0$  equals the number of connected components of $F^*$.

2) Fixed $\tau' \in G\setminus G^0$, let $\tau:=\tau'^2 \in G^0$
and  $\varphi\in \Aut(G^0)$ defined  by $\varphi(h):=\tau' h \tau'^{-1}$.
We define the following action of  $G$ on $G^0\times G^0$: 
\begin{equation}\label{action2}
\begin{split}
g(h_1, h_2) &= (gh_1, \varphi( g)h_2)\\ 
\tau'g(h_1, h_2) &=(\varphi(g)h_2, \tau g \,h_1)
\end{split}\qquad for \,\,g \in G^0
\end{equation}

We define 
$$M:= \left|\bigcup_{g \in G} g\, \im{(\mu')}\right|\,.$$

\end{rem}

\begin{lem}\label{compalb}
Let $S$ be a mixed \qe surface with $q(S)=1$.
Then 
$\deg \psi= \dfrac{|G^0|^2}{M}$.
\end{lem}	

\begin{proof}
Let $u\in E'$.
The action of $G^0 \times G^0$ on $Q^{-1}(u)$ induces a bijection  between $G^0\times G^0$ and $Q^{-1}(u)$;
two points of $Q^{-1}(u)$ belong to the 
same connected component of $F^*$ if and only if the corresponding elements in $G^0\times G^0$  
differ by an element in $\im(\mu')$.\\
Moreover, two points $h,\,h'\in F^*$ map to the same point of $X$ if and only if
 there exists $g\in G$ \sts $g (h')= h$.
So exactly $M$ points of $Q^{-1}(u)$
are mapped into each connected
component of $\varsigma (F^*)$. We conclude since $\deg \psi$ equals  the number of connected components of $F$.
\end{proof}

\begin{prop}\label{galb}
Let $S$ be a mixed \qe surface with $q(S)=1$, then
$$g_{alb}= 1+\dfrac{g(C)-1}{|G^0|^2}M\,.$$
\end{prop}	
	
\begin{proof}
Let us look at diagram (\ref{albdiag}). 
Since $G^0$ is $(1;m_1, \ldots, m_r)$-generated, then $\displaystyle{e(C)=-|G^0|\sum_{i=1}^r\bigg(\frac{m_i-1}{m_i}\bigg)}$.
The $(G^0\times G^0)$-cover $Q$ is branched exactly along the union of $r$ \q{horizontal} copies of $E$ and $r$
\q{vertical} copies of $E$; moreover for each $i$ there are one horizontal copy and one vertical copy 
with branching index $m_i$. Since $E'$ is an elliptic curve that intersects all these copies of $E$ transversally 
in one point, by the Hurwitz's formula applied to $F^*\rightarrow E'$ we obtain 
\begin{equation}\label{Halb}
e(F^*)=-|G^0|^2\sum_{i=1}^r 2\bigg(\dfrac{m_i-1}{m_i}\bigg)\,. 
\end{equation}
On the other hand, the $G$-cover $\varsigma$ is \qe and we get
$$e(F)=\dfrac{e(F^*)}{|G|}=-|G^0|\sum_{i=1}^r \bigg(\dfrac{m_i-1}{m_i}\bigg)= e(C)\,.$$

By Lemma \ref{compalb}, $F$ is the disjoint union of $\deg \psi=\dfrac{|G^0|^2}{M}$ curves of genus $g_{alb}$, and therefore 
$$2-2g(C)=e(F)=\dfrac{|G^0|^2}{M}(2-2g_{alb})\,.$$
\end{proof}	

\section{The minimal model}\label{DTMM}

In this section we want to determine the minimal model of the surfaces we construct. 
We start by recalling some useful results:

\begin{lem}[{\cite[Proposition 1]{Bom}}]
On a smooth surface $S$ of general type every irreducible curve $C$
satisfies $K_S.C\geq -1$.
\end{lem}

\begin{lem}[{\cite[Remark 4.3]{BP10}}]\label{BPsr}
On a smooth surface $S$ of general type every irreducible curve $C$
with $K_S.C\leq 0$ is smooth and rational.
\end{lem}

\begin{prop}\label{234-int}
Let $S$ be a smooth surface of general type.
If $E$, $C$ are distinct smooth rational curves with $E^2=-1$,  $C^2\geq -4$, then $C.E\leq 1$.
\end{prop}

\begin{proof}
Assume by contradiction $C.E\geq 2$. Let $b\colon S\rightarrow S'$ be the
blow-down given by the contraction of $E$ and set $C':=b(C)$.
By the assumption $C.E\geq 2$, $C'$ is singular, so
by Lemma \ref{BPsr}:
\begin{eqnarray*}
0<K_{S'}.C'&=&(K_S-E).(C+(C.E)E)\\
				 &=& (K_S-E).C\\
				 &=& -C^2-2-E.C \leq 0,	
\end{eqnarray*}

\noindent a contradiction.
\end{proof}

\begin{cor}\label{1-2curve}
Let $S$ be a smooth surface of general type.
Assume that $E$ is a $(-1)$-curve in $S$, then
 $E$ intersects at most one $(-2)$-curve.
\end{cor}

\begin{proof}
Suppose $E$ intersects two $(-2)$-curves. By Proposition \ref{234-int}
it intersects each of the $(-2)$-curves  transversally in a point. Then contracting $E$ we get 
two $(-1)$-curves intersecting in a point, which is not possible on a surface of general type. 
\end{proof}

The following is the main result of this section, showing that in the irregular case, the surfaces obtained are automatically 
minimal. 
\begin{thm}\label{minirr}
Let $S$ be an irregular mixed \qe surface of general type, then $S$ is minimal.
\end{thm}
\begin{proof}
Aiming for a contradiction, let $E$ be a $(-1)$-curve on $S$. 

Consider the intermediate quotient $Y=(C \times C)/G^0$, the minimal resolution of its singularities 
$\lambda \colon T \rightarrow Y$ and the involution $\mu$ on $T$ (Lemma \ref{liftinv}). 
Let $\epsilon \colon T' \rightarrow T$ be the blow up of the fixed points of $\mu$. By Proposition \ref{quotbyinv} and its proof, there is a map
 $\tilde{\pi}\colon T'\rightarrow S$ which is a double cover ramified along the exceptional divisors of  $\epsilon$, so branched along a disjoint union of $(-2)$-curves.
 
\noindent Since $E$ can intersect at most one $(-2)$-curve then $\tilde{\pi}^*(E)$ is union of two 
rational curves; let $R$ be one of them. By construction $R$ is not exceptional for the resolution $T' \rightarrow Y$, 
and therefore one of the fibrations $\alpha_i\colon Y\rightarrow C'$ is a surjective map from a rational 
curve to $C'$, contradicting $g(C')=q(S)>0$.
\end{proof}

In the case $q=0$ we borrow an argument of \cite{BP10}. 
Let $\Gamma \subset X=(C\times C)/G$ be a rational curve.
Let $\Gamma':=(\pi \circ \sigma)^*(\Gamma)=\sum_1^k n_i \Gamma_i$ be
the decomposition in irreducible components of its pull back to $C\times C$.
We observe that, since $\pi \circ \sigma$ is quasi-\'etale, $\forall i$ $n_i=1$ 
and that $G$ acts transitively on the set $\{\Gamma_i \mid i=1,\ldots, k\}$.
Hence there is a subgroup $H\triangleleft G$ of index $k$ acting on $\Gamma_1$
\sts $\pi ( \sigma(\Gamma_1))=\Gamma_1 /H=\Gamma$.

\noindent Normalizing $\Gamma_1$ and $\Gamma$, we get the following 
commutative diagram:
$$\xymatrix{
\tilde{\Gamma}_1\ar[r]^\alpha \ar[d]_f & \Gamma_1 \ar[r]^\beta \ar[d]& C\times C\ar[d]\\
\bbP^1\ar[r]^\nu & \Gamma \,\,\ar@{^(->}[r] & X
}$$
Since each automorphism lifts to the normalization, $H$ acts on $\tilde \Gamma_1$
and $f$ is the quotient map $\tilde \Gamma_1\rightarrow \tilde \Gamma_1 /H\cong \bbP^1$.
Moreover $\beta(\alpha(\tilde\Gamma_1))$ is a curve in $C \times C$, and therefore surjects on $C$, hence $g(\tilde \Gamma_1)\geq g(C)\geq 2$ and so $f$ is branched in at least 3 points.

\begin{lem}
Let $p$ be a branch point of $f$, then $\nu(p)$ is a singular point of $X$.
\end{lem}

\begin{proof}
Let $p'\in f^{-1}(p)\subset \tilde \Gamma_1$ be a ramification point of $f$, then $\Stab_H(p'):=H_1\neq\{1\}$ and so
$\Stab_G(\beta(\alpha (p')))\supseteq H_1$. Hence $\nu(f(p'))=\nu(p)\in \Sing(X)$.
\end{proof}

\begin{cor}\label{3point}
Any rational curve in $X$ passes at least 3 times through singular points.
\end{cor}

We will need the following consequence of Proposition \ref{234-int}.

\begin{cor}\label{233}
Let $S$ be a smooth surface of general type.
Assume that $E$ is a $(-1)$-curve in $S$, then $E$ cannot intersect three distinct smooth rational curves with 
self-intersection $-2$ or $-3$.
\end{cor}

\begin{proof}
By Proposition \ref{234-int} $E$ intersects each of the three curves transversally in a point.

Contracting $E$ we get three smooth rational curves with self-intersection $-1$ or $-2$ with a common point. If one of them has self-intersection $-1$, by 
Corollary \ref{1-2curve}  a second curve has self-intersection $-1$, and we find two intersecting $(-1)$-curves, which is impossible on a surface of general type.
 So all have self-intersection $-2$.

We pass to the minimal model of $S$ by contracting all possible $(-1)$-curves. If one of the contracted curves  intersected
one of our three $(-2)$-curves, we get the same contradiction as above. So the image of our 
configuration gives three smooth rational curves with self-intersection $-2$ on a minimal surface of general type with a common point.
This is impossible (see {\it e.g.}, \cite[Proposition 2]{Bom}).
\end{proof}

Proposition \ref{234-int} and Corollary \ref{233} imply that, if the basket of singularities of $X$ is simple enough, 
then $S$ is minimal. More precisely

\begin{prop}\label{minimalitycriterion}
Let $S\rightarrow X$ be a mixed \qe surface of general type. Assume one of the following  
\begin{itemize}
\item[i)] either all exceptional curves for $S\rightarrow X$ have self-intersection $-2$ or $-3$
\item[ii)] or  $\BB(X)=\{2\times C_{4,1}, 3 \times C_{2,1}\}$ 
\end{itemize}
Then $S$ is minimal.
\end{prop}

\begin{proof}
i) In this case, if $S$ were not minimal, by Corollary \ref{3point} and Proposition \ref{234-int}
then there would be a $(-1)$-curve $E$ which intersects three different smooth rational curves with 
self-intersection $-2$ or $-3$, contradicting Corollary \ref{233}.

ii) In this case the exceptional divisor is given by five rational curves which do not intersect each other,
two of self-intersection $-4$ and three of self-intersection $-2$. 
If $S$ were not minimal, by Corollary \ref{3point}, Proposition \ref{234-int} and Corollary \ref{1-2curve}
the dual graph of the resulting configuration of rational curves would be:
\vspace{-.2cm}
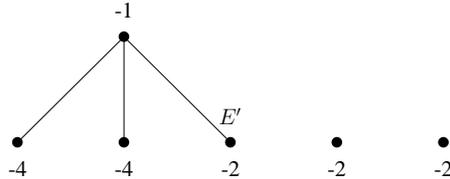
\begin{center}
\scalebox{.7} 
{
\begin{pspicture}(0,-1.5)(8,1.5)
\psdots[dotsize=0.2](0,-1)
\psdots[dotsize=0.2](2,-1)
\psdots[dotsize=0.2](4,-1)
\psdots[dotsize=0.2](2,1)
\psdots[dotsize=0.2](6,-1)
\psdots[dotsize=0.2](8,-1)
\psline[linewidth=0.01cm](0,-1)(2,1)
\psline[linewidth=0.01cm](2,1)(2,-1)
\psline[linewidth=0.01cm](2,1)(4,-1)
\usefont{T1}{ptm}{m}{n}
\rput(2,1.5){-1}
\rput(0,-1.5){-4}
\rput(2,-1.5){-4}
\rput(4,-1.5){-2}
\rput(6,-1.5){-2}
\rput(8,-1.5){-2}
\rput(4,-0.5){$E'$}
\end{pspicture} 
}

\end{center}
After the contraction of the $(-1)$-curve we can also contract $E'$, finding a surface
of general type with two $(-2)$ curves which are tangent in a point. Contracting all possible further 
$(-1)$-curves we find a contradiction as in the end of the proof of Corollary \ref{233} .
\end{proof}

Proposition \ref{minimalitycriterion} is obviously not sharp: we can prove the same result for many
different baskets of singularities by exactly the same argument. We decided to state it in this weak form for 
sake of simplicity, since {\it a posteriori} (inspecting the output of the program we describe in the next section) cases i) and ii) are the only cases that occur for mixed \qe 
surfaces of general type with $p_g=0$ and $K^2>0$.

\section{The classification}\label{FOTC}

We wrote a MAGMA script which computes all mixed \qe surfaces with fixed (input of the script) $p_g$, $q$ and $K^2$.
To write the algorithm we needed to overcome some theoretical problems, namely to find explicit bounds for the basket of singularities $\BB$ and for 
the signatures $(q;m_1,\ldots, m_r)$.

For the basket of singularities, since by Corollary \ref{chiKB} $B(\BB) =24(1-q+p_g)-3K^2$, it is enough to prove
that there are finitely many possible baskets with fixed invariant $B(\BB)$, and show how to  produce the whole list:
\begin{lem}\label{finBas}
Let $B_0\in \bbQ$. Then there are finitely many baskets $\BB$ \sts 
$$B(\BB)=B_0\,.$$
More precisely $|\BB|\leq B_0/3$. Moreover, if $n/a=[b_1,\ldots,b_l]$ then
\begin{enumerate}
	\item $B(C_{n,a}) \geq \sum b_i$;
	\item $B(D_{n,a}) \geq 6+\frac1{2}\sum b_i$. 
\end{enumerate}
\end{lem}
\begin{proof}
We note that $B(C_{n,a})=\frac{a+a'}{n}+\sum b_i \geq 3$,
while $B(D_{n,a})=\frac{B(C_{n,a})}{2}+6 \geq 15/2>3$: this proves $|\BB|\leq B_0/3$,
bounding from above the number of singular points .\\
(1) and (2) are trivial consequences of the definitions of $B(C_{n,a})$ and $B(D_{n,a})$; they show that there are only 
finitely many possible $[b_1,\ldots,b_l]$, so finitely many pairs $(n,a)$.
\end{proof}

The second problem is to bound the possible signatures, once we have fixed $K^2$, $p_g$, $q$ and the basket $\BB$.
We have to find upper bounds for $r$ and for the $m_i$.

\begin{defi}\label{TBX}
Let $\rho\colon S\rightarrow X=(C\times C)/G$ be  a mixed \qe surface.
Let $(q;m_1,\ldots, m_r)$ be the signature of the induced generating vector for $G^0$. 
Let $\BB$ be the basket of singularities of $X$.
 Then we define the following numbers:
\begin{equation*}
\Theta:=2q(S)-2+\sum_{i=1}^r\bigg(\frac{m_i-1}{m_i}\bigg)\,,
\end{equation*}

\begin{equation*}
\beta:= \frac{12 \chi(\OO_S)+k(\BB)-e(\BB)}{3\Theta}\,,
\end{equation*}

\end{defi}

\begin{defi}[{see \cite{YPG}}]
The minimal positive integer $I_x$ \sts $I_xK_X$ is Cartier in a neighborhood of $x\in X$ is called
the \no{index of the singularity} $x$.
The \no{index} of a normal variety $X$ is the minimal positive integer $I$ \sts $IK_X$ is Cartier.
In particular, $I=\lcm_{x\in \Sing(X)} I_x$ depends only on the basket of singularities.
\end{defi}

The index of a singularity of type $C_{n,a}$ is 
$$I_x=\frac{n}{\gcd(n,a+1)}\,.$$

We can now give the bounds we need.

\begin{prop}\label{ineq}
Let $\rho\colon S\rightarrow X=(C\times C)/G$ be  a mixed \qe surface.
Let $(q;m_1,\ldots, m_r)$ be the signature of the induced generating vector for $G^0$. Let $\BB=\BB_C\cup \BB_D$ be the basket of singularities of $X$. Then
\begin{itemize}
	\item[a)] $\Theta >0$ and $\beta=g(C)-1$;
	\item[b)] $r\leq 2\Theta+4(1-q)$;
	\item[c)] $m_i \leq 4\beta +6$;
	\item[d)] each $m_i$ divides $2 \beta I$ where $I$ is the index of $Y$;
	\item[e)]  $m_i\leq \dfrac{2I\beta\Theta+1}{M}$, with $M:=\max\big\{\frac{1}{6}, \frac{r-3+4q}{2}\big\}$;
	\item[f)] except at most $|\BB_C|+|\BB_D|/2$ indices $i$,
         $m_i \leq \dfrac{\beta\Theta+1}{M}$ and divides $\beta$.
\end{itemize}
\end{prop}

\begin{proof}
a) Since $q(S)=g(C/G^0)$, by Hurwitz's formula:
$$2(g(C)-1)=|G^0|\cdot \Theta \,,$$
hence $\Theta=\dfrac{2(g(C)-1)}{|G^0|}>0$, since $g(C)\geq 2$.
Let $k:=k(\BB)$  and $B:=B(\BB)$. By Corollary \ref{chiKB} and  Proposition \ref{Ksquare}
 we get
$$\beta=\frac{24\chi+3k-B}{6\Theta}=\frac{K_S^2+k}{2\Theta}=\frac{8(g(C)-1)^2}{4\Theta |G^0|}=g(C)-1\,.$$

b) By definition $\Theta\geq 2q-2+\frac{r}{2}$, hence $r\leq 2\Theta -4(q-1)$.

c) Since $m_i=\ord(h_i)$ and $h_i$ is an automorphism of a curve of genus $g\geq 2$, 
by Wiman's Theorem (see \cite{wim}) $m_i\leq 4g+2=4\beta+6$.

d) Since $|\BB(Y)|=2|\BB_C|+|\BB_D|$, the claim follows by \cite[Proposition 1.14, d]{BP10}.

e) We first show 
$$
\Theta+\frac{1}{m_i}\geq\max\left\{\frac{1}{6},\frac{r-3+4q}{2}\right\}.
$$
Since $\displaystyle{\Theta=2q-2+r -\sum_{j=1}^r\frac{1}{m_j}}$, 
we get
$$
\Theta+\frac{1}{m_i}=2q-2+r-\sum_{j\neq i}\frac{1}{m_j} \geq 2q-2+r -\frac{r-1}{2}=\frac{r-3+4q}{2}\,.
$$
Since $\Theta>0$, $\frac{r-3+4q}{2} \geq \frac16$ unless $r = 3$ and $q=0$. In this case, $\Theta>0$ implies that
 at most one $m_i$ can be equal to $2$. Hence also in this case $\Theta+\frac{1}{m_i} \geq 0-2+3-\sum_{j\neq i}\frac{1}{m_j}\geq 1-\frac12 - \frac13= \frac{1}{6}$.

By d), we get $$\left(\max\left\{\frac{1}{6},\frac{r-3+4q}{2}\right\}\right)m_i\leq 1+\Theta \cdot m_i\leq 1+2I\beta\Theta.$$

f) By \cite[Proposition 1.14, e]{BP10}, except for at most $|\BB(Y)|/2=|\BB_C|+|\BB_D|/2$ indices, $m_i$ divides 
$\beta$. From $m_i\leq \beta$, it follows that
$$\left(\max\left\{\frac{1}{6},\frac{r-3+4q}{2}\right\}\right)m_i\leq 1+\Theta m_i\leq 1+\Theta \beta\,.$$
\end{proof}

We used the inequalities proved in this section to produce an algorithm to compute all mixed \qe surfaces with fixed
$p_g, q$ and $K^2$, following the same strategy of the algorithm of \cite{BP10} which computed the product-quotient surfaces
with $p_g=q=0$ (input was just $K^2$). The algorithm uses also the following simple remarks:

\begin{rem}\label{|G|}
By Hurwitz's formula $|G|=2|G^0|=\dfrac{4(g(C)-1)}{\Theta}=\dfrac{4\beta}{\Theta}$.
\end{rem}
\begin{rem}
Let $\rho\colon S\rightarrow X=(C\times C)/G$ be  a mixed \qe surface.
Let $(q;m_1,\ldots, m_r)$ be the signature of the induced generating vector for $G^0$.
If $X$ has a singular point of type  $C_{n,a}$ or $D_{n,a}$, then there exists 
$m_i$ \sts $n$ divides $m_i$.

Indeed, the singular point is the class of a point $(x,y) \in C \times C$
\sts $\Stab_{G^0}(x,y)=\tr{\eta}$ with $o(\eta)=n$.
 $x$ is a ramification point of $c\colon C \rightarrow C/G^0$,
and its ramification index, that equals $|\Stab_{G^0}(x)|$, is one of the $m_i$. 
Since $\eta\in \Stab_{G^0}(x)$, it follows that $n$ divides $m_i$.
\end{rem}

\bigskip

We explain here very briefly the strategy of the algorithm.

Having fixed the values of $K^2_S$, $p_g(S)$ and $q(S)$, by Corollary \ref{chiKB} we know $B(\BB)$, and
Lemma \ref{finBas} gives easily a procedure to produce the finite list of baskets with that invariant $B$. Then, for each basket, we produce the finite list of all signatures $(q;m_1,\ldots,m_r)$ respecting all conditions 
in Proposition \ref{ineq},  including the requirement that $\beta$ is an integer. 

Now, for each basket and for each associated signature, the orders of $G$ and $G^0$ are computed by 
Remark \ref{|G|}. Then the script checks all the finitely many groups $G^0$ of that order, 
and their unsplit degree $2$ extensions $G$. 

Then we have a list of quintuples (basket, signature, $G^0$, generating vector, extension), each quintuple 
gives a family of mixed \qe surfaces (just determined by ($G^0$, generating vector, extension) as explained in 
Remark \ref{algdata}), and all mixed \qe surfaces with the prescribed invariants are here.
Anyway, in this list there are also surfaces with different invariants: those whose singularities does 
not correspond to the basket. Then the script  
 computes these singularities in each case, using the results of section \ref{TBOS}, 
(in particular Propositions \ref{s-h} and \ref{fixi}), and it discards the surfaces with wrong basket.

Moreover, different generating vectors give isomorphic surfaces if they differ by some 
{\it Hurwitz moves}, which are described, in the cases we need, in \cite[Section 5]{Penegini}. 
The script computes this action on the remaining generating vectors, and returns only a representative for each 
orbit. Finally, the script computes, using a result by Armstrong (\cite{Arm65}, \cite{Arm68}), the
fundamental groups (see \cite{Frap11}) of the resulting surfaces.

Our code skips some signatures giving rise to groups of large order, either not covered by
the MAGMA SmallGroup database, or causing extreme computational complexity.
The program returns the list of the skipped cases, which have to be studied separately.

A commented version of the full program can be downloaded from:
$$\mbox{\url{http://www.science.unitn.it/~pignatel/papers/Mixed.magma}}\,$$

Using it, we proved Theorems A, B and C as follows.

{\it Sketch of the proof of Theorems A, B and C.}
By Corollary \ref{chiKB} every mixed \qe surface has $K^2 \leq 8\chi$; so the possible invariants of a minimal 
surface of general type with $\chi=1$ are $K_S^2=1,2,3,4,5,6,7,8$ and, by Beauville's inequality \cite{be82} $p_g \geq 2q-4$, 
$p_g=q=\{0,1,2,3,4\}$. We ran  our program for all these values; it returned the surfaces in Tables 
\ref{tabBigInt0}, \ref{tabBigInt1} and \ref{tabBigInt2}.

\begin{table*}[!ht]\small
\centering
\hspace{-1cm}
\begin{minipage}{0.5\textwidth}
\begin{tabular}{|c|c|c|c|}
\hline
$K^2_S$ & $\Sing X$ & Sign. & $|G^0|$\\
\hline
1 & $2 \times C_{8,1}, C_{4,1}$ &   2,3,8  & 6336 \\
1 & $3\times C_{4,1}, C_{4,3}$ &   2,3,8  & 2304 \\
1 & $C_{8,1},C_{4,1},C_{8,5}$ &   2,3,8  & 4032 \\
1 & $4\times C_{4,1},C_{2,1}$ &   2,3,8  & 2880 \\
1 & $2\times C_{8,3},C_{4,1},C_{2,1}$ &   2,3,8  & 2304 \\
1 & $2\times C_{2,1},C_{8,3},C_{8,1}$ &   2,3,8  & 3744 \\
\hline
2& $2\times C_{8,3}, C_{4,1}$ &   2,3,8   &2880\\
2& $C_{8,3},C_{8,1},C_{2,1}$ &   2,3,8   &4320\\
2& $4\times C_{4,1}$ &   2,4,5   &2400\\
2& $4\times C_{4,1}$ &   2,3,8   &3456\\
2& $C_{8,3},C_{8,5},C_{2,1}$ &   2,3,8   &2016\\
2& $2\times C_{4,1}, 3\times C_{2,1}$ &   2,3,8   &2304\\
2& $2\times C_{4,1}, C_{3,1},C_{3,2}$ &   2,3,8   &2496\\
\hline
3& $2\times C_{4,1}, 2\times C_{2,1}$ &   2,3,8   &2880\\
3& $C_{8,3},C_{8,1}$ &   2,3,8   &4896\\
3& $2\times C_{4,1},C_{5,3}$ &   2,4,5   &2160\\
3& $C_{8,3},C_{8,5}$ &   2,3,8   &2592\\
3& $ C_{4,3},C_{4,1},C_{2,1}$ &   2,3,8   &2304\\
\hline
\multicolumn{4}{c}{}
\end{tabular}
     \end{minipage}
     \begin{minipage}{0.45\textwidth}
\begin{tabular}{|c|c|c|c|}
\hline
$K^2_S$ & $\Sing X$ & Sign. & $|G^0|$\\
\hline
4 & $C_{4,3},C_{4,1}$ &   2,3,8  & 2880 \\
4 & $4 \times C_{2,1}$ &   2,3,8  & 2304 \\
4 & $C_{3,1},C_{3,2},C_{2,1}$ &   2,3,8  & 2496 \\
4 & $2 \times C_{4,1}, C_{2,1}$ &   2,4,5  & 2400 \\
4 & $2 \times C_{4,1}, C_{2,1}$ &   2,3,8  & 3456 \\
\hline
5& $C_{5,2},C_{2,1}$ &   2,4,5   &2160\\
5& $3\times C_{2,1}$ &   2,3,8   &2880\\
5& $C_{3,1},C_{3,2}$ &   2,3,8   &3072\\
5& $2\times C_{4,1}$ &   2,4,5   &2800\\
5& $2\times C_{4,1}$ &   2,3,8   &4032\\
\hline
6 & $2 \times C_{2,1}$ &   2,4,5   & 2400\\
6 & $2 \times C_{2,1}$ &   2,3,8   & 3456\\
6 & $2 \times C_{5,3}$ &   2,4,5   & 2560\\
\hline
7 & $C_{2,1}$ &   2,3,9   & 2268  \\
7 & $C_{2,1}$ &   2,4,5   & 2800  \\
7 & $C_{2,1}$ &   2,3,8   & 4032  \\
\hline
8 & $\emptyset$ &   2,3,9   & 2592\\
8 & $\emptyset$ &   2,4,5   & 3200\\
8 & $\emptyset$ &   2,3,8   & 4608\\
\hline
\end{tabular}
     \end{minipage}
\caption{The skipped cases for $p_g=q=0$ and $K^2 > 0$}
\label{tabSkip}
\end{table*}

As mentioned, the surfaces returned by the program may be not all mixed \qe surfaces with the required invariants, since the program 
is forced to skip some signatures, giving rise to groups of large order. The program returns the list of these ``skipped'' cases.
 
For the cases $p_g=q\not =0$, this list is empty, so  the Tables 
\ref{tabBigInt1} and \ref{tabBigInt2} are complete. We report the list of the ``skipped'' signatures for $p_g=q=0$ in Table \ref{tabSkip}.
We proved that none of these cases occur by arguments very similar to the analogous 
proofs in the papers \cite{BCGP08, BP10, Frap11} and therefore we do not include them here. The interested reader 
will find the details in $$\mbox{\url{http://www.science.unitn.it/~pignatel/papers/skipped.pdf}}$$

Now let us consider the  surfaces in Tables \ref{tabBigInt0}, \ref{tabBigInt1} and \ref{tabBigInt2}. A surface with $K^2>0$ is either of general type or rational, 
therefore regular and simply connected: a quick inspection of the tables shows that this latter case does not occur, so all constructed surfaces are of general type.
 By Theorem \ref{minirr} and Proposition \ref{minimalitycriterion} all the constructed surfaces are minimal. Moreover, again by Proposition \ref{minirr}, 
since every minimal surface of general type has positive $K^2$, we have found all irregular mixed \qe surfaces with $p_g=q$.
\qed


\

\noindent \textbf{Authors Adresses}:

Davide Frapporti: University of Bayreuth, Lehrstuhl Mathematik VIII;\\
Universitaetsstrasse 30, D-95447 Bayreuth, Germany\\
E-mail address: \url{Davide.Frapporti@uni-bayreuth.de}

\

Roberto Pignatelli: Dipartimento di Matematica, Universit\`a di Trento;\\
Via Sommarive 14; I-38123 Trento, Italy\\
E-mail address: \url{Roberto.Pignatelli@unitn.it}

\end{document}